\documentclass[12pt]{amsart}
\usepackage{amssymb,amsfonts,latexsym,amscd}

\usepackage[all,cmtip,2cell]{xy}
\UseAllTwocells
\usepackage{verbatim}
\newtheorem{theorem}{Theorem}[section]
\newtheorem{lemma}[theorem]{Lemma}
\newtheorem{proposition}[theorem]{Proposition}
\newtheorem{corollary}[theorem]{Corollary}
\theoremstyle{definition}
\newtheorem{definition}[theorem]{Definition}

\newtheorem{example}[theorem]{Example}
\newtheorem{question}[theorem]{Question}
\newtheorem{conjecture}[theorem]{Conjecture}
\newtheorem{remark}[theorem]{Remark}

\newcommand{\C}{\mathcal{C}}

\newcommand{\ben}{\begin{enumerate}}
\newcommand{\een}{\end{enumerate}}

\begin{document}

\title[Lower central series of associative algebras]{On properties of the lower central series of associative algebras}

\author{Nabilah Abughazalah}
\address{Mathematical Sciences Department
Princess Nourah bint Abdulrahman University
P.O.Box 84428, Riyadh 11671, Saudi Arabia}
\email{nhabughazala@pnu.edu.sa}

\author{Pavel Etingof}
\address{Department of Mathematics, Massachusetts Institute of Technology,
Cambridge, MA 02139, USA} \email{etingof@math.mit.edu}

\begin{abstract}
We give an accessible introduction into the theory of lower central series of associative algebras,
exhibiting the interplay between algebra, geometry and representation theory that is characteristic for this subject, 
and discuss some open questions. In particular, we provide shorter and clearer proofs of the main results 
of this theory. We also discuss some new theoretical and computational results and conjectures 
on the lower central series of the free algebra in two generators modulo a generic homogeneous 
relation.
\end{abstract}

\maketitle

\section{Introduction} 

Let $A$ be an associative algebra. Let $L_1(A)=A$, and $L_{i+1}(A)=[A,L_i(A)]$ 
for $i\ge 1$; so, 
$$
L_1\supset L_2\supset\dots \supset L_m\supset\dots
$$ 
is the {\it lower central series} of $A$ regarded as a Lie algebra. 
In 2006, in the pioneering paper \cite{FS}, Feigin and Shoikhet studied
$L_i$ in the case when $A=A_n$ is the free algebra in $n$ generators over a field of characteristic zero. 
More precisely, they studied the successive quotients $B_i:=L_i/L_{i+1}$, and discovered that they have a rich structure:
starting from $i=2$, they are finite length modules over the Lie algebra $W_n$ of polynomial vector fields 
in $n$ variables. This allows one to say a lot about the structure of $B_i$ (for example, 
compute its Hilbert series in a number of cases, and prove the surprising fact that $B_i$ has polynomial growth for $i\ge 2$). 

Since then, the theory of lower central series of associative algebras was developed in a number of papers, e.g., 
\cite{DKM, EKM, DE, AJ, BoJ, BB, Ke, BJ, BEJKL, JO, CFZ, KL, FX}. In particular, the papers \cite{EKM,Ke,JO,CFZ} 
studied the ideals $M_i=AL_i$ and their quotients $N_i=M_i/M_{i+1}$, and showed that 
they have a similar (and in some ways simpler) structure to $L_i,B_i$. 

The main goal of this paper is to give an accessible introduction into the theory of lower central series of associative algebras,
exhibiting the interplay between algebra, geometry and representation theory that is characteristic for this subject, 
and to discuss some open questions. In particular, we provide shorter and clearer proofs of the main results 
of this theory, and also discuss some new results.  
 
This paper is based on lecture notes by the first author of the lectures delivered by the second author at MIT in the Fall of 2014. 
Many of these results were obtained by undergraduate and high school students and their mentors 
in the MIT research programs UROP, RSI, SPUR, and PRIMES. One of the goals of this paper is to provide 
a gentle entry for students who want to work in this field. 

The organization of the paper is as follows. In Section 2, we discuss preliminaries and give a review of the main results about the lower 
central series. In Sections 3-10 we provide proofs of these results. Finally, in Section 11, we discuss 
the case of algebras with relations and give some new results, computational data, and conjectures. 
\\

 {\bf Acknowledgments.}   The work of P.E. was partially supported by the NSF grant DMS-1000113. 
Both authors gratefully acknowledge the support of Aramco Ibn Khaldun fellowship. The first author wants to thank Princess Nourah bint Abdulrahman University. We are grateful to Rumen Dangovski,
Darij Grinberg, Gus Lonergan and Nate Harman for useful comments. Finally, we are very grateful to Eric Rains for providing MAGMA 
programs for the calculations described in the last section. 

\section{Overview of the main results} 

\subsection{Preliminaries} 

\subsubsection{The lower central series of a Lie algebra} 

Let $R$ be a commutative ring (for example, a field). 

Let $A$ be a Lie algebra over $R$. 
Define a series of Lie ideals in $A$ inductively:
 $$ L_1(A)=A;$$  $$L_{2}(A)=[A, L_1(A)]=[A,A];$$ 
$$L_{3}(A)=[A, L_{2}(A)]=[A,[A,A]];$$ $$\vdots $$ $$L_{i+1}(A)=[A, L_{i}(A)]=\underbrace{[A,[A,[A,[A\dots[A,A]\dots]]]]}_{i +1\ \text{times}}$$ where the bracket $[C,D]$ of two $R$-submodules $C$ and $D$ is defined to be the span of elements $[c,d]$ such that $c\in C$, $d\in D$. This series is called the {\it lower central series} of $A$. We abbreviate $L_i(A)$ as $L_i.$ We have $$A=L_1\supset L_2\supset L_3\supset \dots.$$

We define the successive quotients of the lower central series by $B_i:= L_i/ L_{i+1}$.

Since $A=L_1\supset L_2\supset...$ is a Lie algebra filtration, the direct sum $B:=B_1\oplus B_2\oplus B_3\oplus...$ 
is a graded Lie algebra, i.e. $[B_i,B_j]\subset B_{i+j}$.
Moreover, it is easy to check that the Lie algebra $B$ is generated by $B_1$. 

\subsubsection{The lower central series of an associative algebra}

Now let $A$ be an associative unital algebra over $R$. 
In the future, we will drop the word "unital" and say just "associative algebra" or "algebra". 
Define a bracket operation on $A$ by $[a,b]=a\cdot b-b\cdot a.$ This operation makes $A$ into a Lie algebra over $R$. 
This allows us to define the lower central series $L_i=L_i(A)$ and its successive quotients $B_i$. 

Note that $$B_1=A/[A,A]=HH_0(A)=HC_0(A),$$ (the zeroth Hochschild and cyclic homology of $A$). 
So $B_i$ for $i\geq 2$ may be viewed as some higher analogues of this.

Denote the two-sided ideals generated by each $L_i$ by $M_i$, i.e.  $M_i:=A\cdot L_i\cdot A$. They also form a filtration 
$$
A=M_1\supset M_2\supset M_3\supset\dots.
$$ 
It is easy to check that $M_i=A\cdot L_i$.

The motivation for considering $M_i$ is that $A/M_i$ is the maximal quotient of $A$ which is Lie nilpotent of nilpotency degree $i$, i.e $$[\dots [a_1,a_2],\dots ,a_i]=0\ \  \forall a_1,a_2 \dots,a_i\in A/M_i.$$ This is an important special case of a polynomial identity in an algebra. 

We also define the successive quotients $N_i:=M_i/M_{i+1}$.

For example, $N_1=A_{\operatorname{ab}}$, the abelianization of the algebra $A$, obtained by taking the quotient of $A$ by the relation $[a,b]=0$. 
The $R$-modules $B_i$ and $N_i$ will be the main objects of study in this paper. 

 \begin{example}
1. Let $A=A_n=A_n(R)=R<x_1,x_2,\dots,x_n>$ be the free non-commutative algebra in $n$ generators.
$A_n$ is a free $R$-module, with an $R$-basis formed by all possible monomials, or words, 
in the letters $x_1,...,x_n$. So, we have $$\underset{i\geq1}{\bigcap}L_i=\underset{i\geq1}{\bigcap}M_i=0.$$ Indeed, $L_i$, $M_i$ are graded by length of words and so when we intersect, the minimal possible degree goes to infinity. 
 So, the spaces $B=\oplus_{i\ge 1}B_i$ and $N=\oplus_{i\ge 1}N_i$ can serve as a "graded approximation" to $A$. 
 This will be one of the main examples in this paper. 

2. Let $f_1,...,f_m$ be homogeneous elements in $A_n$. Let 
$$
A=A_n/<f_1,f_2,\dots,f_m>.
$$
 Then $\underset {i\geq1} {\bigcap} L_i= \underset{i\geq1}{\bigcap}M_i=0$
 for the same reason as above.

3. If $A$ is a commutative algebra then $M_i=L_i=N_i=B_i =0$ for $i\geq2$.

4. If $A=\mathbb{C}<x,y>/<yx-xy-1>$ (the Weyl algebra) then $A=[A,A]$ (as $x^iy^j=[y,\frac{x^{i+1}y^j}{i+1}]$), 
so $L_i=A$ and $B_i=N_i=0$ for all $i$.

 Hence, the last two examples will not be interesting for us in this paper. 
 \end{example}
 
 \begin{definition}
 Let $M=\underset{i\geq0}{\bigoplus} M[i]$ be a graded $R-$module such that ${\rm rank}_R(M[i])<\infty$. The Hilbert series of $M$ is defined to be
 $$
 h_M(t)=\sum_{i=0}^\infty {\rm rank}_R(M[i])t^i.
 $$
For instance, if $R$ is a field, we have 
 $$
 h_M(t)=\sum_{i=0}^\infty \dim(M[i])t^i.
 $$
 \end{definition}
 
\begin{example} If $M=A_n$ graded by length of words then ${\rm dim}(M[i])=n^i$, so $h_M(t)=\frac{1}{1-nt}$.  
\end{example}
 
 \subsection{Results on the lower central series of general algebras} 
 
 \begin{theorem}\label{Jenn}(\cite{J})
  If $A$ is a finitely generated algebra then for all $k$ there exists $m$ such that $M_2^m\subset M_k$.
  This $m$ depends only on the number of generators of $A$ and on $k$. 
 \end{theorem}
 
 A proof of Theorem \ref{Jenn} over $\Bbb C$ will be given in Section 5, see Corollary \ref{Jenn1}
 (and essentially the same proof goes through over any base ring).  
 
 \begin{theorem}(\cite{GL})
 \label{thmguptalevin}
  $M_i\cdot M_j\subset M_{i+j-2},$ for every $i,j\geq 2$.
   \end{theorem}
   
A proof of Theorem \ref{thmguptalevin} is given in Section 4. 
 
 \begin{definition}
 Let $I$ be an ideal of an algebra $A$. We say that $I$ is:
 \begin{enumerate}
 \item  nil if for all $ x \in I$ there exists $N$ such that $x^N=0$.
 \item nilpotent if there exists $N$ such that $ I ^N=0.$
 \end{enumerate} 
 \end{definition}
 
 \begin{corollary} If $A$ is Lie nilpotent then 
 $M_2$ is a nil ideal. Moreover, $M_2$ is nilpotent if $A$ is also finitely generated.
 \end{corollary}
 
 \begin{proof}
Assume that $A$ has nilpotency degree $k$, i.e. $M_k=0$. Let $x\in M_2$. 
Then $x\in M_2(A')$ for some finitely generated subalgebra $A'\subset A$. 
By Theorem \ref{Jenn}, there exists $m$ such that $M_2(A')^m\subset M_k(A')$. 
But $M_k(A')\subset M_k(A)=0$. Hence $x^m=0$, and $M_2$ is nil. If $A$ is in addition finitely generated, then 
by Theorem \ref{Jenn} there is an $m$ such that $M_2^m\subset M_k=0$, so $M_2$ is nilpotent.   
 \end{proof} 
 
\begin{theorem}[\cite{BJ}]
\label{thmBapat}
 If $i$ or $j$ is odd and $\frac {1}{6 } \in R$ then $M_i M_j\subset M_{i+j-1}.$
\end{theorem}

Theorem \ref{thmBapat} is proved in Section 3. 

\begin{conjecture}
If $i,j\geq 2$ are both even then in general $M_i \cdot M_j \nsubseteq M_{i+j-1}$.
\end{conjecture}

One of the important corollaries of Theorem \ref{thmBapat} is the following theorem. 
Suppose $A$ is generated by $x_1,...,x_n$, and let $A_{\le 2}$ be the span of elements 
$1,x_i,x_ix_j$. 

\begin{theorem} \label{thmBapat1} (\cite{AJ}) 
Let $A$ be an algebra over a field of characteristic zero. For $m\ge 2$ we have 
$$
B_{m+1}=[A_{\le 2},B_m]+\sum_{i,j,k\text{ distinct}}[x_i[x_j,x_k],B_m].
$$
\end{theorem}

This theorem is proved in Section 9. 

\begin{remark} Later Bapat and Jordan showed 
that the last summand is in fact redundant, i.e.  one has the following stronger (but more difficult) 
theorem (conjectured in \cite{AJ}). 

\begin{theorem}\label{thmBapat2} (\cite{BJ})
Retain the assumptions of Theorem \ref{thmBapat1}. Then for $m\ge 2$, one has $B_{m+1}=[A_{\le 2},B_m]$.  
\end{theorem} 

We refer the reader to \cite{BJ} for a proof of this result. 
\end{remark} 

\subsection{Results on the lower central series of the free algebra $A_n(\mathbb{C})$}

Let $A=A_n$ and $R=\mathbb{C}$. We are interested in the Hilbert series of $N_i$ and $B_i$. 

First consider the case $i=1$. Recall that $N_1=A_{ab}=\mathbb{C}[x_1,x_2,\dots,x_n]$, so $h_{N_1}(t)=\frac{1}{(1-t)^n}$. 
Also, recall that $B_1=A/[A,A]$. Therefore, it has a basis consisting of cyclic words (necklaces) in $x_1,x_2,\dots,x_n$, i.e. words
considered up to cyclic permutation. So, ${\rm dim}B_1[d]$ 
is equal to the number of necklaces of length $d$. Let us denote this number by $a_d(n)$. 
It follows from Polya's enumeration theorem that 
$$
\prod_{d\geq 1}(1-t^d)^{a_d(n)}=(1-nt)(1-nt^2)(1-nt^3)\dots,
$$
which allows us to easily compute $a_d(n)$ recursively.  
Hence, $a_{d}  \sim n^d$ (in sense that $\underset{d\to\infty}{\lim} \frac{\log (a_d(n))} {d}=\log n),$
and in particular $a_d(n)$ has exponential growth as $d\to\infty$. 

It turns out, however, that the dimensions of the homogeneous parts $B_i[d]$ of the spaces $B_i$
grow polynomially for all $i\ge 2$ (and the same holds for $N_i$). Namely, the  
spaces $B_2,\ N_2$ are known explicitly, and we'll discuss their structure and Hilbert series below.  
The Hilbert series of $B_i,N_i$ for $i\ge 3$ are known only in a few special cases, 
but we have the following theorem. 

\begin{theorem}[\cite{FS,DE,AJ}]\label{hilserthm}
The Hilbert series $h_{B_i}(t)$ and $h_{N_i}(t)$ for $i\geq 3$ are of the form $\frac{P(t)}{(1-t)^n}$, where $P(t)$ is a polynomial with positive integer coefficients.
\end{theorem}

Theorem \ref{hilserthm} is proved in Section 10. 

As we mentioned, this property fails for $B_1$, but this is ``corrected" by the following proposition. 

\begin{proposition}\label{central} (\cite{FS}) Let $A$ be any algebra over $\Bbb C$, $Z\subset B_1=A/[A,A]$ 
be the image of $M_3$ in $B_1$, i.e.   $Z=M_3/M_3 \cap L_2$, and $\overline{B_1}=B_1/Z.$ Then 

(1) $Z\subset B_1$ is central in $B$ 
(so that $\overline{B}:=B/Z=\overline{B}_1\oplus B_2\oplus...$ is a graded Lie algebra); 

(2) for $A=A_n$, we have $h_{\overline {B}_1}(t)=\frac{P(t)}{(1-t)^n}$, where $P(t)$ is a polynomial.
\end{proposition}

Proposition \ref{central}(1) is proved in Section 4. Proposition \ref{central}(2) follows from Theorem \ref{fsthm}(5)
below. 

Now let us describe $B_2$ and $N_2$ in the case of the free algebra $A=A_n(\Bbb C)$. 
We will get this description from the explicit description of $A/M_3$, obtained by Feigin and Shoikhet (\cite{FS}). 
Namely, let $\Omega=\Omega(\mathbb{C}^n)$ be the space of differential forms on $\mathbb{C}^n$
with polynomial coefficients. We have a decomposition $\Omega=\oplus_{k=0}^n \Omega^k$, 
where $\Omega^k$ is the space of differential forms of degree $k$, which is a free 
module over $\Bbb C[x_1,...,x_n]$ with basis $dx_{i_1}\wedge...\wedge dx_{i_k}$, $i_1<...<i_k$. 
Let $\Omega^{\rm even}=\oplus_{0\le i\le n/2}\Omega^{2i}$ be the space of even forms. 
The space $\Omega$ is a supercommutative algebra under wedge product, 
and $\Omega^{\rm even}$ is a commutative subalgebra in it.
Also, $\Omega$ carries a de Rham differential $d: \Omega^k\to \Omega^{k+1}$. 

Following Feigin and Shoikhet \cite{FS}, we introduce another product 
on $\Omega$ (and $\Omega^{\rm even}$), denoted $\ast$. 
Namely, we set 
$$
\alpha \ast \beta=\alpha \wedge \beta +\frac{1}{2}d\alpha\wedge d\beta;$$ 
the corresponding commutator on $\Omega_{\rm even}$ is given by
$$[\alpha,\beta]_\ast=\alpha\ast\beta-\beta\ast\alpha=d\alpha\wedge d\beta,$$
if $\alpha,\ \beta\in \Omega_{\rm even}$ (below we will denote this commutator just by $[,]$ when no confusion is possible). 
It is easy to check that this product is associative, but, unlike the wedge product, it is noncommutative
on even forms, and
does not preserve the grading by degree of differential forms. 
We will denote the algebra $\Omega$ with this operation by $\Omega_{\ast}$, and 
its even part by $\Omega_\ast^{\rm even}$. Note that $\Omega=\Omega_\ast$ as vector spaces, and 
we will use these two notations interchangeably when we don't consider multiplication. 

\begin{proposition}\label{polid} 
The algebra $\Omega_\ast^{\rm even} $ satisfies the polynomial identity $[[a,b],c]=0$. 
\end{proposition}

\begin{proof}
Since 
$$
[a,b]=da\wedge db=d(a\wedge db),
$$
we have $d[a,b]=0$, and hence  
$$
[[a,b],c]=d[a,b]\wedge dc=0.
$$ 
\end{proof}

\begin{corollary}
\label{corhomo}
 The algebra homomorphism $\widetilde{\phi}: A\to \Omega_\ast^{\rm even}$ defined by $ \widetilde{\phi}(x_i)=x_i$
descends to a homomorphism $\phi: A/M_3\to \Omega_\ast^{\rm even}$ .
\end{corollary}

\begin{proof}
This immediately follows from Proposition \ref{polid}. 
\end{proof} 

One of the main theorems about lower central series of associative algebras is the following theorem
about the lower central series of $A_n(\Bbb C)$.  

\begin{theorem}\label{fsthm}(\cite{FS}) 
\begin{enumerate}
\item $\phi $ is an algebra isomorphism.
\item $\phi(N_2)=\Omega^{\rm even,+}=\oplus_{1\le i\le n/2}\Omega^{2i}$, the space of even forms of positive degree.
\item The image of $L_2$ in $A/M_3$ is naturally isomorphic to $B_2$ (i.e.  $L_2\cap M_3=L_3$). 
\item $\phi(B_2)=\Omega_{{\rm exact}}^{\rm even}$, the space of even exact forms, i.e.  of even forms $\alpha$ such that $d\alpha=0$.
\item $\phi $ induces a map $\hat\phi: B_1\to \Omega ^{\rm even}/\Omega_{{\rm exact}}^{\rm even}$ such that $\hat\phi(Z)=0$, so $\hat\phi$ defines 
a map $\overline{\phi}: \overline{B}_1\to \Omega ^{\rm even}/\Omega_{{\rm exact}}^{\rm even}$. The map $\overline{\phi}$ is an isomorphism.
\end{enumerate}
\end{theorem}

Parts (1),(2),(4) and (5) of Theorem \ref{fsthm} are proved in Section 6, and part (3) 
(which is the hardest one) is proved in Section 8. 

Using Theorem \ref{fsthm}, it is easy to compute the Hilbert series of $\overline{B}_1$, $B_2$ and $N_2$ for $A=A_n$.

Let us now present some bounds on the degree of the polynomial $P(t)$ in the numerator of the Hilbert series 
of $B_m,N_m$ for $m\ge 3$ (see Theorem \ref{hilserthm}). 

\begin{theorem}[\cite{AJ}]\label{hsbounds} For $m\ge 3$ we have  
$$\text{deg}\big(h_{B_m}(t)(1-t)^n\big)\leq 2m-3+2\lfloor{\frac{n-2}{2}\rfloor};$$
$$\text{deg}\big(h_{N_m}(t)(1-t)^n\big)\leq 2m-2+2\lfloor{\frac{n-2}{2}\rfloor}.$$
\end{theorem}

These bounds are proved using representation theory.
Namely, let $W_n$ be the Lie algebra of polynomial vector fields on $\mathbb{C}^n.$
In more detail, $W_n$ consists of expressions of the form 
$v=\sum_i f_i(x_1,...,x_n)\partial_n$, 
where $f_i\in \Bbb C[x_1,...,x_n]$, 
with commutator defined by 
$$
[\sum_i f_i\partial_i,\sum_j g_j\partial_j]=\sum_{i,j}(f_i\frac{\partial g_j}{\partial x_i}\partial_j-g_j\frac{\partial f_i}{\partial x_j}\partial_i).
$$

\begin{theorem}\label{wnact}
\begin{enumerate}
\item (\cite{FS}) $\overline{B}_1(A_n)$ and $B_i(A_n),\ i\ge 2$ are naturally graded modules over the Lie algebra $W_n$.
\item (\cite{EKM}) $N_i(A_n)$ are naturally graded modules over $W_n$.
\end{enumerate}
\end{theorem}

Let $\lambda$ be a partition with at most $n$ parts. To this partition one can attach the 
polynomial $GL(n)$-module $V_\lambda$, and we can define the 
module ${\mathcal F}_\lambda:=V_\lambda\otimes \Bbb C[x_1,...,x_n]$ over 
$W_n$ of tensor fields on $\Bbb C^n$ of type $\lambda$ (see Example \ref{vecfiel} below).
Note that $\mathcal{F}_\lambda$ is graded by the eigenvalues of 
the Euler field $E:=\underset{i}{\sum} x_i\frac{\partial}{\partial x_i}$, and  

\begin{equation}
\label{charFlambda}
h_{\mathcal{F}_\lambda}(t)=\frac{{\rm dim} V_\lambda t^{|\lambda|}}{(1-t)^n},
\end{equation} 

We will see below in Theorem \ref{Rud} that 
$\mathcal{F}_\lambda$ is irreducible if and only if $\lambda_1\geq 2$, or $\lambda=(1^n)$.

The following result is proved in the papers \cite{AJ,BJ,Ke}: 

\begin{theorem} 
For all $m\geq 3,\ B_m$ and $N_m$ are of finite length as modules over $W_n$, 
with all composition factors being $\mathcal{F}_\lambda$ with $|\lambda|\ge 2$. 
Moreover, if $\mathcal{F}_\lambda$ occurs in $B_m$ then $|\lambda|\leq 2m-3+2\lfloor{\frac{n-2}{2}\rfloor}$, and if 
$\mathcal{F}_\lambda$ occurs in
 $N_m$ then $|\lambda|\leq 2m-2+2\lfloor{\frac{n-2}{2}\rfloor}.$ 
\end{theorem} 

This theorem together with formula \eqref{charFlambda} implies Theorem \ref{hsbounds}. 

We also have the following interesting result about arbitrary finitely generated algebras $A$ over a field,
proved independently by Jordan and Orem \cite{JO} and by Cordwell, Fei, and Zhu \cite{CFZ}:  

   \begin{theorem}\label{JorOr}
  Suppose that ${\rm Spec}(A_{ab})$ is at most $1$-dimensional and has finitely many non-reduced points. 
Then $B_m$ and $N_m$ for $m\ge 2$ are finite dimensional.
    \end{theorem}
    
For the proof, we refer the reader to \cite{JO,CFZ}. 

\section{Proof of Theorem \ref{thmBapat}}

We will now prove Theorem \ref{thmBapat}. To this end, we introduce some lemmas, 
which can be found in the paper \cite{BJ}.

For brevity introduce the notation 
$$
[a_1,a_2,\dots, a_m]:=[a_1,[a_2,\dots[a_{m-1},a_m]]\dots].
$$

\begin{lemma}
\label{lemmaBapat1}
If $1/6\in R$ then $[M_3,A]\subset L_4.$
\end{lemma}

\begin{proof}
 Let $[x[y,z,u],v]\in [M_3,A] $ for $A=A_5$ which is generated by $x,y,z,u,v$. Our job is to show that $[x[y,z,u],v]\in L_4.$
First note that $[x[y,z,u],v]= [x\ast[y,z,u],v]$ modulo $L_5$ where $a\ast b=\frac{1}{2}(ab+ba)$. Then we have the following:
\begin{equation}
[x\ast [y,z,u],v]+[y\ast[x,z,u],v]=[[x\ast y,z,u],v]\ \in L_4;
\end{equation}
\begin{equation}
[x\ast [y,z,u],v]+[v\ast [y,z,u],x]=-[x\ast v,y,z,u] \in L_4.
\end{equation}
This means that the expression $[x\ast [y,z,u],v]$ is anti-symmetric under the group $G=S_3\{x,y,v\}\times S_2\{z,u\}$ as an element of $A/L_4$. Now we have  the following identity, which may be verified directly on coefficients of the $10$ monomials in $x,y,z,u,v$ which represent $G$-orbits on all the monomials having degree 1 in each variable:
\begin{equation}
\text{Alt}_G[x\ast[y,z,u],v]=\text{Alt}_G(4[z\ast x,y,v,u]-2[x,z,y,u\ast v]),
\end{equation}
where $\text{Alt}_G$ is antisymmetrization with respect to $G$. 
So we conclude that $[x[y,z,u],v]\in L_4.$
\end{proof}

\begin{lemma}\label{lemmaBapat2}
If $1/6\in R$ then $[M_3,L_k]\subset L_{k+3}.$
\end{lemma}

\begin{proof}
Consider $[a[b,c,d],[x,y]]$, where $x\in L_1=A,\ y\in L_{k-1}.$ We have
 \begin{equation}
[a[b,c,d],[x,y]]=[[a[b,c,d],x],y]+[x,[a[b,c,d],y]].
\end{equation}
 We prove the statement by induction on $k$. The statement is true when $k=1$ by Lemma \ref{lemmaBapat1}. Now consider $k>1$ and assume that the statement is true for $k-1$. Notice that the first term in the RHS of $(5)$ is in $L_{k+3}$ by Lemma \ref{lemmaBapat1} and the second term is in $L_{k+3}$ by the inductive hypothesis. Therefore, $[a[b,c,d],[x,y]]\in L_{k+3}$ as required.
 \end{proof}
 
 The following lemma is standard and follows easily from the Jacobi identity. 
 
 \begin{lemma}\label{multilin} Any polylinear element of the free Lie algebra in $N$ generators $w_1,w_2,\dots,w_N$ (i.e.  one of degree 1 in every generator) can be written as a linear combination of 
 $$
 [w_{\sigma(1)},\dots,w_{\sigma(N-1)},w_N],\sigma \in S_{N-1}.
 $$ 
 \end{lemma} 
 
 \begin{corollary}\label{multilin1} A polylinear element of $L_{N-1}(A_N)$ is a linear combination of $N-1$-tuple commutators, where all entries are single letters, except the last one, which is a product of two letters. 
 \end{corollary} 
 
 \begin{proposition}
 \label{proposition}
 If $\frac{1}{6} \in R$ then $[M_j,L_k]\subset L_{k+j}$ whenever $j$ is odd.
 \end{proposition}
 \begin{proof}
The proof is by induction in $j$, running over odd values of $j$. For $j=1$ the statement is trivial, so 
we just have to justify the induction step, from $j-2$ to $j$.  Consider a general element $[m,l]$ with $m\in M_j$ and $l=[l_1,l_2,\dots,l_k]\in L_k$. Write $m=a[b,c,d],\ d\in L_{j-2}, \ a,b,c\in A$. 
 Let us regard $a,b,c,d,l_1,...,l_k$ as independent variables (in particular, we regard $d$ as a variable instead of an iterated bracket). 
 Then by Lemma \ref{lemmaBapat2}, $[a[b,c,d],l]$ is a polylinear element of $L_{k+3}(A_{k+4}),$ where $A_{k+4}$ is the free algebra generated by  $a,b,c,d,l_1,l_2,\dots,l_k$.  By Corollary \ref{multilin1}, 
this element is a linear combination of terms of the form 
\begin{equation}\label{threeterms}
[y_1,\dots,y_{k+2},y_{k+3}d],\
[y_1,\dots,y_{k+2},dy_{k+3}],\
[y_1,\dots,d,\dots,y_{k+1},y_{k+2}y_{k+3}]
\end{equation}
for various permutations $(y_1,\dots,y_{k+3})$ of $(a,b,c,l_1,\dots,l_k)$. We now plug in $d=[d_1,\dots,d_{j-2}]$. Then terms of the third type in (\ref{threeterms}) lie in $L_{k+j}$ by definition. 
Also, by the induction assumption $[y_{k+2},y_{k+3}d]\in L_{j-1}$, hence $[y_{k+2},dy_{k+3}]\in L_{j-1}$. Plugging this into the terms of the first two types in (\ref{threeterms}), we find that they are also in $L_{k+j}$.
This justifies the step of induction. 
 \end{proof}
 
\begin{remark} Here is another version of the proof of Proposition \ref{proposition}, due to Darij Grinberg. 
 
 Let us consider a more general setting. Let $A$ be an algebra, and $D$ be a Lie ideal of $A$. We define $L_i(D)\subset A$ for all $i \geq 1$ recursively by
$L_1(D) = D$; $L_{i+1}(D) = [A, L_i(D)]$. (So $L_i(D) = [A, [A, ..., [A, D] ...] ]$, with $i$ being the total number of $A$'s plus $1$.) 
We reserve the notation $L_i$ for $L_i(A)$. 

Notice that $L_i(L_j(D)) = L_{i+j-1}(D)$ for all $i$ and $j$. It is also 
easy to see that all $L_i(D)$ are Lie ideals of $A$ and satisfy $[L_i(D), L_j] \subseteq L_{i+j}(D)$.

We also set $M_i(D) = AL_i(D) = L_i(D)A$ (the last equality holds because $L_i(D)$ is a Lie ideal of $A$). Clearly, $M_i(D)$ is an ideal of $A$.

From now on, we assume that $D$ is an ideal (not just a Lie ideal) of $A$. (``Ideal" always means ``two-sided ideal" here.) Here is a slight generalization of Lemma \ref{lemmaBapat2}:

\begin{lemma} \label{lemmaBapat1a} If $1/6 \in R$, then $[M_3(D), A] \subseteq L_4(D)$.
\end{lemma}

\begin{proof} 
The proof is the same as the proof of Lemma \ref{lemmaBapat2}, up to adding "$(D)$"s after the $L$'s and $M$'s. The trick is that every iterated commutator of $4$ elements of $A$ lies in $D$ as long as at least one of these $4$ elements lies in $D$, and that $D$ is an ideal under the $\ast$-product (because it is an ideal under the usual product).
\end{proof} 

Next, we generalize Lemma \ref{lemmaBapat2}:

\begin{lemma} \label{lemmaBapat2a} If $1/6 \in R$, then $[M_3(D), L_k] \subseteq L_{k+3}(D)$.
\end{lemma}

\begin{proof} The proof again is the same as that of Lemma \ref{lemmaBapat2}, using Lemma \ref{lemmaBapat1}, but now $d$ belongs to $D$.
\end{proof} 

The proof of Proposition \ref{proposition} now proceeds by induction over odd $j$ as before, but the induction step becomes simpler to explain. 
Namely, let $D = M_{j-2}$. This is clearly an ideal of $A$. Also, $L_{j-2} \subseteq D$. Now,
$M_j = A L_j = A [A, [A, L_{j-2}]] \subseteq A [A, [A, D]]$ (since $L_{j-2} \subseteq D$)
$= M_3(D)$, whence
$[M_j, L_k] \subseteq [M_3(D), L_k] \subseteq L_{k+3}(D)$ (by Lemma \ref{lemmaBapat2a}) 
$= L_{k+2} ([A, D]) \subseteq L_{k+2} (L_{j-1})$ (since $[A, D] = [D, A] = [M_{j-2}, L_1] \subseteq L_{j-1}$ by the induction hypothesis)
$= L_{k+j}$.
\end{remark} 

 \begin{proof}[Proof of Theorem \ref{thmBapat}]
 
 We may assume that $k$ is odd. Clearly it is enough to show that $L_jL_k\subset M_{j+k-1}.$ Let $a\in A,\ x\in L_{j-1},\ y\in L_k$. Then we have 
 $$
 [x,a]y=[x,ay]-a[x,y].
 $$
 The LHS is a completely general generator of $L_jL_k$. By Proposition \ref{proposition}, the first term of the RHS is in $M_{j+k-1}$. 
 Also, it is clear that the second term of the RHS is in $M_{j+k-1}$. This implies the theorem. 
 \end{proof} 
 
\begin{remark} 
A computer calculation with polylinear elements of $A_5$ shows that 
Lemma \ref{lemmaBapat1} holds over $R=\Bbb Z[\frac{1}{3}]$. Hence Theorem \ref{thmBapat} holds when $\frac{1}{3}\in R$, with the same proof. 
However, as shown by Krasilnikov \cite{Kr}, Theorem \ref {thmBapat} and Lemma \ref{lemmaBapat1} fail over $\Bbb Z$ and over fields of characteristic $3$. 
\end{remark}  
 
\section{Proof of Theorem \ref{thmguptalevin} and Proposition \ref{central}(1)}

\subsection{Proof of Theorem \ref{thmguptalevin}}

For the proof, we need the following lemma, which can be proved by using the Leibniz rule.

 \begin{lemma}
 \label{lemmaleibniz} One has 
$$
[c_1,c_2,\dots,c_n,ab]=\underset{j}{\sum}\underset{\sigma\in T_j}{\sum}[c_{\sigma(1)},\dots,c_{\sigma(j)},a][c_{\sigma(j+1)},\dots,c_{\sigma(n)},b]
$$ 
where $T_j$ is the set of permutations 
$\sigma \in S_n$ such that 
$\sigma(1)<\dots<\sigma(j)$ and $ \sigma(j+1)<\dots<\sigma(n)$.
 \end{lemma}

 It suffices to show that $L_jL_k\subset M_{j+k-2}$. The proof is by induction on $k$. Assume that for any $2\leq s<k,\ L_jL_s\subset M_{j+s-2}$. So our job is to show that $[r,a_j][b_1,\dots,b_k]\in M_{j+k-2}$, where $r=[a_1,\dots,a_{j-1}]$. Consider the expression 
$$
T:=[b_1,\dots,b_{k-2},r,a_j[b_{k-1},b_k]].
$$ 
By Lemma \ref{lemmaleibniz}, this is the sum of the following expressions:
 \begin{enumerate}
 \item $[r,a_j][b_1,\dots,b_k]$.
 \item $[b_{\sigma(1)},\dots,b_{\sigma(l)},r,a_j][b_{\sigma(l+1)},\dots,b_{\sigma(k-2)},b_{k-1},b_k]$, $l>0$. This is in $M_{j+k-2}$ by the inductive hypothesis.
 \item $[b_{\sigma(1)},\dots,b_{\sigma(l)},a_j][b_{\sigma(l+1)},\dots,b_{\sigma(k-2)},r,b_{k-1},b_k]$. This is contained in $L_{l+1}\cdot L_{k+j-l-1}$. Since $l\le k-2$, we have $l+1<k$, so this is also in $M_{k+j-2}$ by the induction hypothesis for the opposite ring $A^{\rm op}$. 
\end{enumerate}

So it remains to show that 
$T\in M_{j+k-2}$. But 
$$
T=-[b_1,\dots,b_{k-2},a_j[b_{k-1},b_k],r]\ \in L_{j+k-2}.
$$

\subsection{Proof of Proposition \ref{central}(1)}

\begin{proposition}[\cite{FS}]
\label{propfeigin}
Over any ring $R$ we have $[M_3,L_j]\subset L_{j+2}$.
\end{proposition}

 \begin{proof}
 It can be shown easily that $[A,M_3]\in L_3$ by just checking that $[x[y,z,u],v]\in L_3$ for $A=A_5$ (this can be done by hand or using a computer).
 Now, using Lemma \ref{multilin}, we get $[M_3,L_j]\subset [A...[A,M_3]..]$ ($j$ commutators), and the latter 
 is in $L_{j+2}$, since $[A,M_3]\subset L_3$. 
 \end{proof}
 
 \begin{corollary} (Proposition \ref{central}(1))
 Let $Z\subset B_1$ be the image of $M_3$ in $B_1$. Then $Z$ is central in the graded Lie algebra $B=\underset{i\geq1}{\bigoplus}B_i$.
 \end{corollary}
 
 \begin{proof}
 By Proposition \ref{propfeigin} we have 
that $[M_3,A]\subset L_3$ and hence we get $[B_1,Z]=0$, as the image of $L_3$ in $B_2$ is $0$. But $B_1$ generates $B$ as a Lie algebra, so $[B,Z]=0$.
 \end{proof}
 
 Therefore, we can define the Lie algebra $\overline{B}=B/Z$, 
where $\overline{B}=\overline{B}_1\oplus \underset{i\geq 2}{\bigoplus}B_i$ and $\overline{B_1}=B_1/Z.$  

\section{Proof of Theorem \ref{fsthm} (1),(2),(4),(5)}

\subsection{Proof of Theorem \ref{fsthm}(1)}
Recall that we have a new non-commutative product on differential forms $\Omega=\Omega(\Bbb{C}^n)=\underset{k\geq 0}{ \bigoplus} \Omega^k(\Bbb{C}^n)$, given by:
$$\alpha\ast\beta=\alpha\wedge\beta+\frac{1}{2}d\alpha\wedge d\beta;$$
the corresponding commutator is given by
$$[\alpha,\beta]_\ast=\alpha\ast\beta-\beta\ast\alpha=d\alpha\wedge d\beta,$$
if $\alpha,\ \beta\in \Omega^{\rm even}$. We have by Corollary \ref{corhomo} a homomorphism $\widetilde{\phi}: A\to \Omega_\ast^{\rm even}$ defined by $\widetilde{ \phi}(x_i)=x_i$, and it 
descends to a homomorphism $\phi: A/M_3\to \Omega_\ast^{\rm even}$.

\begin{proposition} (Theorem \ref{fsthm}(1))
 $\phi $ is an isomorphism.
 \end{proposition}
 \begin{proof}
 First we show surjectivity. Note that $\phi([x_i,x_j])=dx_i\wedge dx_j$. This implies that 
 $$
 \phi(x_{1}^{m_1}\cdots x_n^{m_n} [x_{i_1},x_{i_2}]\cdots [x_{i_{2r-1}},x_{i_{2r}}])=
 $$
 $$
 x_{1}^{m_1}\cdots x_n^{m_n} dx_{i_1}\wedge dx_{i_2}\wedge \cdots \wedge dx_{i_{2r-1}}\wedge dx_{i_{2r}}+\text{higher degree terms}.
 $$
But such elements span $\Omega_\ast^{\rm even}$, so $\phi$ is surjective.

Now we show injectivity. Observe that for any algebra $A$ and any $x,y,z,t\ \in A $, we have modulo $M_3$:

$\begin{array}{lcl}
[x,y][z,t]+[x,z][y,t]&=& [x,y[z,t]]+[x,z[y,t]]\\
&=& [x,[yz,t]]-[x,[y,t]z]+[x,z[y,t]]\\
&=&[x,[yz,t]]+[x,[z,[y,t]]]\\
&=& 0\ (\text{mod}\ M_3).
\end{array}$

In particular, $[x,y][x,t]=0\ \text{mod}\ M_3$. Hence the elements $u_{ij}=[x_i,x_j]\in A/M_3$ satisfy $u_{ij}u_{il}=0,\ u_{ij}u_{kl}+u_{kj}u_{il}=0$ where $i,j,k,l$ are distinct. Also $[x_i,u_{jk}]=0,$ so $u_{jk}$ is central in $A/M_3$. So the elements $x_1^{m_1}\cdots x_n^{m_n} u_{i_1i_2}\cdots u_{i_{2r-1}i_{2r}},$ for all possible $m_1,\cdots,m_n\geq0$ and $1\leq i_1<i_2<i_3<\cdots<i_{2r}\leq n$ form a spanning set for $A/M_3$. But the images of these elements under $\phi$ are linearly independent, so $\phi$ is injective and thus an isomorphism.
\end{proof}

\subsection{Proof of Theorem \ref{fsthm}(2)}

\begin{corollary}(Theorem \ref{fsthm}(2)) $\phi(N_2)=\Omega^{\rm even,+}=\oplus_{1\le i\le n/2}\Omega^{2i}$, the space of even forms of positive degree.
\end{corollary}

\begin{proof} Note that $N_2$ is the span of the elements $a[b,c];\ a,b,c\in A/M_3$.
Let $\alpha,\beta,\gamma\in\Omega^{\rm even}$ and let $a=\phi^{-1}(\alpha),\ b=\phi^{-1}(\beta),\ c=\phi^{-1}(\gamma)\in A/M_3$
(this makes sense since $\phi$ is an isomorphism by Theorem \ref{fsthm}(1)). Then $\phi([b,c])=d\beta \wedge d\gamma,$ so $\phi(a[b,c])=\alpha\wedge d\beta\wedge d\gamma$. But $\Omega^{\rm even,+}$ is spanned by such elements.
\end{proof}

\begin{corollary}\label{Jenn1}(Theorem \ref{Jenn} over $\Bbb C$)
If $A$ is an algebra over $\Bbb{C}$ with $\leq 2r-1$ generators then $M_2^r\subset M_3$. Hence $M_2^{(k-2)r}\subset M_k$ for $k\ge 3$.
\end{corollary}

\begin{proof}
It suffices to assume that $A=A_n,\ n\leq 2r-1$. Then, by Theorem \ref{fsthm}(2), $\widetilde{\phi}(M_2)\subset \Omega^{\rm even,+}\subset \Omega^+.$ Therefore $\widetilde{\phi}(M_2^r)\subset \Omega^{\geq 2r}=0$, which implies that $M_2^r\subset M_3$. But by Theorem \ref{thmguptalevin}, 
$M_3M_i\subset M_{i+1}$ for all $i$, which by induction in $k$ yields $ M_2^{(k-2)r}\subset M_k$ for all $k\geq 3$.
\end{proof}

\subsection{Proof of Theorem \ref{fsthm}(4,5)}

\begin{corollary} (Theorem \ref{fsthm}(4,5))
\label{cor3.4}
(i) $\widetilde{\phi}(L_2)=\Omega_{\rm exact}^{\rm even}=\{\omega\in \Omega^{\rm even},\  \omega=d\alpha\}.$
 
 (ii) The map $\widetilde{\phi}$ induces a map $\hat{\phi}: B_1\rightarrow \Omega^{\rm even}/\Omega_{\rm exact}^{\rm even}$ which kills $Z$, so defines a map $$\overline{\phi}:\overline{B}_1\rightarrow \Omega^{\rm even}/\Omega_{\rm exact}^{\rm even},$$ which is an isomorphism.
 \end{corollary}
 
 \begin{proof}
 (i) $L_2$ is spanned by $[a,b],\ a,b\in A$. We have $$\widetilde{\phi}([a,b])=d\widetilde{\phi}(a)\wedge d\widetilde{\phi}(b)=d(\widetilde{\phi}(a)\wedge d\widetilde{\phi}(b))$$ which is exact. Moreover, fix $b$ such that $\phi(b)=w$, a given element of $\Omega^{\rm even}$ of degree $2r$. Then $\phi([b,x_i])=dw\wedge dx_i=d(w\wedge dx_i).$
 But elements $w\wedge dx_i$ span the space of odd forms. This implies (i). 
 
 (ii) follows immediately from (i). 
 \end{proof}

\section{Proof of Theorem \ref{wnact}}
Let $A$ be any algebra and let ${\rm Der}(A)$ be the Lie algebra of derivations of $A$. 
Let $D\in {\rm Der}(A)$ be such that $D(A)\subset M_3.$ 

\begin{proposition}\label{dli}
  For all $i\geq 2,\  D(L_i)\subset L_{i+1} .$
  \end{proposition}
  
  \begin{proof} By the Leibniz rule we have
$$
D[a_1,a_2,\cdots ,a_i]=\sum\limits_{j=1}^{i} [a_1,\cdots,[Da_j,\cdots,a_i]].
$$
  Using the Jacobi identity, this can be written as a linear combination of elements of the form 
 $[a_{\sigma(1)},...,[a_{\sigma(i-1)},Da_{\sigma(i)}]]$, where $\sigma\in S_i$.  
But we know that $[a_{\sigma(i-1)},Da_{\sigma(i)}]\in L_3$ since $[M_3,A]\subset L_3$
by Proposition \ref{propfeigin}. 
  Hence $D[a_1,\cdots ,a_i]\in L_{i+1}.$
  \end{proof}
  
  \quad  So we get $D|_{B_i} =0$ for all $i\geq 2$ and $D|_{\overline{B}_1} =0$. Also we have $$D(M_i)=D
  (AL_i)\subset D(A)L_i+ AD(L_i),$$ but $D(A)L_i\subset M_3L_i\subset M_{i+1}$ by Theorem \ref{thmguptalevin} and $AD(L_i)\subset AL_{i+1}=M_{i+1}$ by Proposition \ref{dli}. 
  Therefore, $D(M_i)\subset M_{i+1}.$ So $D|_{N_i}=0$ for all $i\ge 1$.
  
  Thus, we have established the following proposition. 
  
  \begin{proposition}
  The action of the Lie algebra ${\rm Der}(A)$ on $\overline{B}_1$, $B_i$ for $i\ge 2$, and $N_i$ for $i\ge 1$ factors through the image of 
  the natural map $\psi: {\rm Der}(A)\to {\rm Der}(A/M_3)$.
  \end{proposition}
  
  \begin{proposition}
  The map $\psi: {\rm Der}(A_n)\to {\rm Der}(A_n/M_3)$ is surjective.
  \end{proposition}
  
  \begin{proof}
  Let $\delta \in {\rm Der}(A_n/M_3)$ be a derivation, and let $\delta(x_i)=\alpha_i$. Let $a_i\in A_n$ be such that $\widetilde{\phi}(a_i)=\alpha_i$, and let $D=\sum a_i\frac{\partial}{\partial x_i}$, which means $D(x_i)=a_i$. Then $\psi(D)=\delta.$
  \end{proof}
  
  So we have an action of ${\rm Der}(A/M_3)={\rm Der}(\Omega_\ast^{\rm even})$ on $\overline{B}_1, B_i$ for $i\ge 2$, and $N_i$. 
  
  Now note that $\ast$ is an invariant operation, so $W_n\subset {\rm Der}(\Omega_{\ast}^{\rm even} ).$ This gives an action of $W_n$ on $\overline{B}_1, B_i $, $i\ge 2$, and $N_i, i\geq 1$, 
  and proves Theorem \ref{wnact}.
  
  \begin{remark}
We leave it to the reader to check that the action of $W_n$ on $\overline{B}_1=\Omega^{\rm even}/\Omega^{\rm even}_{\rm exact},\ N_1=\Bbb{C}[x_1,\cdots, x_n]$ and $N_2=\Omega^{\rm even,+}$ is the standard action.
  \end{remark}

\section{Representations of $W_n$} 

Now let us discuss the representation theory of $W_n$. 
 
 \begin{definition} 
 A module $M$ over $W_n$ belongs to the category $\mathcal{C}$ if the operators $x_i\frac{\partial}{\partial x_i}\in W_n$ on $M$ are semisimple with finite dimensional joint eigenspaces, and their eigenvalues are in $\Bbb{Z}_{\geq0}$.  
\end{definition} 

\begin{example}\label{vecfiel}
\begin{enumerate}
\item $M=\Bbb{C}[x_1,x_2,\dotsc,x_n]$, with the tautological action of vector fields $v=\sum f_i(x_1,...,x_n)\frac{\partial}{\partial x_i}$. Then $M$ is graded by $\Bbb{Z}_{\geq0}^n$: 
one has $M=\underset{i_1,\dotsc,i_n\geq 0}{\bigoplus}M[i_1,\dotsc,i_n]$. This is a grading 
by eigenvalues of $x_i\frac{\partial}{\partial x_i}$, so $M\in \mathcal{C}$. 

\item {\it Tensor field modules.} This is a generalization of the previous example. 
Namely, let $V=(\Bbb{C}^n)^\ast$, and let $T_N=SV\otimes V^{\otimes N}$ be the space of tensors of type $(N,0)$ on $V^\ast =\Bbb{C}^n$. 
This space has a basis $x_1^{m_1}\dotsc x_n^{m_n}dx_{i_1}\otimes\dotsc \otimes dx_{i_N}.$ The vector field 
$v=\sum_i f_i(x_1,...,x_n)\frac{\partial}{\partial x_i}$ acts by the Leibniz rule, as follows: 

$\begin{array}{lcl}
v(\phi dx_{i_{1}}\otimes\dotsc \otimes dx_{i_{N}})=v(\phi)dx_{i_{1}}\otimes\dotsc \otimes dx_{i_{N}}\\
+\phi dv(x_{i_{1}})\otimes dx_{i_{2}}\otimes\dotsc
 \otimes dx_{i_{N}}
 +\dotsc
  +\phi dx_{i_1}\otimes\dotsc \otimes dv(x_{i_N}).
\end{array}$
Note that $T_0=\Bbb C[x_1,...,x_n]$, the module from the previous example, and 
$T_1=\Omega^1$, the module of 1-forms. 
\end{enumerate}
\end{example}

The modules $T_N$, in general, are not irreducible or even indecomposable. 
Namely, $T_N$ carries an action of the symmetric group 
$S_N$ which commutes with $W_n$, and therefore 
decomposes into isotypic components according to the type of the $S_N$-symmetry. 
Let us consider this decomposition in more detail. 

Recall that irreducible representations of $S_N$ are labeled by partitions of $N$. 
If $\lambda$ is a partition of $N$, let $\pi_{\lambda}$ denote the corresponding representation of $S_N$. 
Then we can define the $W_n$-module 
$$
\mathcal{F}_\lambda:={\rm Hom}_{S_N}(\pi_\lambda, T_N)=SV\otimes S_\lambda(V),$$
where $S_\lambda(V):={\rm Hom}_{S_N}(\pi_\lambda,V^{\otimes N})$ is the corresponding 
Schur functor. Clearly $\mathcal{F}_\lambda\neq 0 \iff$ length$(\lambda)\leq n$. 

 \begin{example}
 Suppose that $\lambda=(1^N),\ N\leq n.$ Then $\mathcal{F}_\lambda=\Omega^N(\Bbb{C}^n).$ In this case $\mathcal{F}_\lambda$ is not irreducible if $N\neq n$. Indeed we have an exact sequence $$
 0\to \Omega^N_{\rm closed}\to \Omega^N\to \Omega^{N+1}_{\rm exact}\to 0, $$
 where $\Omega^{N+1}_{\rm exact}=\Omega^{N+1}_{\rm closed}$. Note that $\Omega^n_{\rm closed}=\Omega^n={\mathcal F}_{1^n}$. 
 \end{example}
 
 \begin{proposition}(\cite{R})
 $\Omega_{\rm closed}^N$ is irreducible for all $N=0,\dotsc,n$.
 \end{proposition}
 
 \begin{theorem}\label{Rud}(\cite{R})
 \begin{enumerate}
 \item $\mathcal{F}_\lambda$ is irreducible unless $\lambda= 1^N$, $N<n$.
 \item Any simple object in $\mathcal{C}$ is isomorphic to $\mathcal{F}_\lambda( \lambda\neq 1^N)$ or to $\Omega _{\rm closed}^N( N=0,\dotsc,n)$,
 while these modules are pairwise non-isomorphic. 
 \end{enumerate}
 \end{theorem}
  
We showed above that $\overline{B}_1,\ B_i(i\geq2)$ and $N_i(i\geq1)$ are $W_n-$ modules. It is easy to see that they belong to $\mathcal{C}$. Hence they admit a composition series with successive quotients $\mathcal{F}_\lambda$, $\Omega_{\rm closed}^ N$, where each $\mathcal{F}_\lambda$ or $\Omega_{\rm closed}^ N$ has finite multiplicity. 

\begin{proposition}\label{finlength}  The $W_n$-modules $\overline{B_1}$, $B_m$ for $m\ge 2$, and $N_m$ have finite length.  
\end{proposition} 

\begin{proof} Let $0\ne Y\in \C$, and let $d_i=\dim Y[i]$, the dimension of the subspace of elements of degree $i$ in $Y$. 
It is easy to see from Theorem \ref{Rud} that $d_i\sim \frac{Cn^i}{i!}$ as $i\to\infty$, where $C\ge 1$ is an integer.
Therefore, it suffices to show that 
$\dim B_m[i]\le C_mn^i$ for some $C_m>0$ and sufficiently large $i$. But this follows by induction in $m$ using Theorem \ref{thmBapat2}
from the case $m=2$, which is obtained  from Theorem \ref{fsthm}(4). Also, the statement about $\overline{B}_1$ follows 
from Theorem \ref{fsthm}(5). 

To prove the statement for $N_m$, note that by Theorem \ref{thmBapat1}, 
$N_m=AB_m=\sum_s A[c_s,B_{m-1}]$, where $c_s$ runs through the elements $x_i,x_ix_j$, and $x_i[x_j,x_k]$. 
Now, $a[c_s,b]=[ac_s,b]-[a,b]c_s$, so we get $N_m\subset B_m+\sum_sB_mc_s$. This means that 
 $\dim N_m[i]\le \widetilde{C}_mn^i$ for large $i$ (as we have a similar bound for $B_m$). 
This implies that $N_m$ have finite length.  
\end{proof}

\section{Proof of Theorem \ref{fsthm}(3)}

\subsection{Proof of Theorem \ref{fsthm}(3) for $n=2$}

 Now we will prove Theorem \ref{fsthm}(3) in the special case $n=2$. Recall that Theorem \ref{fsthm}(3) states that $L_2\cap M_3=L_3$, or, equivalently, 
 the natural surjective map 
 $$
 \theta: B_2\to \Omega^{\rm even}_{\rm exact}
 $$
is an isomorphism. In the case $n=2$, $\Omega^{\rm even}_{\rm exact}=\Omega^2=
 \Bbb C[x,y]dx\wedge dy$, which is a bigraded space with a 1-dimensional space in bidegree $(i,j)$ for all $i,j\ge 1$ and zero everywhere else. So, since $\theta$ is surjective, it suffices to prove the following proposition. 
 
\begin{proposition} If $A=A_2$ then $B_2$ is spanned by the (images of) the elements $[y^i,x^j]$. 
\end{proposition}
 
 \begin{proof} 
 Note that in the free algebra generated by $x,a$, we have: 
 
 $\begin{array}{lcl}
 [a,x^m]&=&ax^m-x^ma\\
 &=&ax^m-xax^{m-1}+xax^{m-1}-x^2ax^{m-2}+\dotsc+x^{m-1}ax-x^ma\\
 &=&[ax^{m-1},x]+[xax^{m-2},x]+\dotsc+[x^{m-1}a,x].
 \end{array}$ 
 
 Modulo $L_3$ all these summands are the same, so we have 
 $$
 [a,x^m]=m[x^{m-1}a,x]
 $$
 modulo $L_3$, and hence 
 $$
 [x^{m-1}a,x]=\frac{1}{m}[a,x^m]
 $$
 modulo $L_3$. More generally, 
 
 $\begin{array}{lcl}
 [x^ka,x^r]&=&x^kax^r-x^{k+r}a\\
           &=&x^kax^r-x^{k+1}ax^{r-1}
 +x^{k+1}ax^{r-1}-x^{k+2}ax^{r-2}+\dotsc\\
 &+&x^{k+r-1}ax-x^{k+r}a
 =
 
 [x^{k-1}ax^r,x]+\dotsc+[x^{k+r-1}a,x],
 \end{array}$
 which modulo $L_3$ equals $r[x^{k+r-1}a,x]$ (since all the $r$ summands are the same modulo $L_3$). So we get 
 $$
 [x^ka,x^r]=\frac{r}{k+r}[a,x^{k+r}]
 $$
 modulo $L_3$. 
 
 Now we proceed to prove the proposition. Clearly, $L_2$ and hence $B_2$ is spanned by elements of the form 
 $[a,x]$ and $[a,y]$, $a\in A$. So it suffices to show that for any $a\in A$, 
 $[a,x]$ is a linear combination of $[y^i,x^j]$ modulo $L_3$ (the case of $[a,y]$ is similar). 

To this end, it is enough to show that for any monomial $b\in A$ and $r\ge 1$, 
 $[b,x^r]$ is a linear combination of $[y^i,x^j]$ modulo $L_3$. We will prove this statement by induction in the degree $d$
 of $b$. The base case ($d=1$) is clear. To justify the step of induction, let $b$ have degree $d$. 
 If $b'$ is obtained by cyclic permutation of $b$, then $[b,x]=[b',x]$ modulo $L_3$, so we can replace 
 $b$ with $b'$. Also, if $b=y^d$, we are done. Otherwise, there exists a cyclic permutation $b'$ of $b$
 such that $b'=x^ka$, $k>0$. Then, as shown above, 
 $$
 [b,x^r]=[x^ka,x^r]=\frac{r}{k+r}[a,x^{k+r}]
 $$
 modulo $L_3$. Since ${\rm deg}(a)<d$, we are done by the induction assumption. 
This justifies the induction step and proves the proposition. 
\end{proof}

\subsection{Proof of Theorem \ref{fsthm}(3) for all $n$}
 
\begin{lemma}
\label{lemkerspan}
Let $\zeta:\Omega\otimes \Bbb{C}^n\to \Omega$ defined by $\zeta(\sum\limits_{i}\alpha_i\otimes e_i)=\sum\limits_{i}d\alpha\wedge dx_i$ where $e_i$ is the standard basis of $\Bbb{C}^n$. Then ker$\zeta$ is spanned by the following:
\begin{enumerate}
\item $b\otimes e_i$, where $\ b\in \Omega_{\rm closed};$
\item Elements $\beta\wedge dx_i\otimes e_i$ and $\ \beta\wedge dx_j\otimes e_i+\beta\wedge dx_i\otimes e_j,\ \beta \in \Omega;$
\item Elements $ \sum \limits_{i}\frac{\partial f}{\partial x_i}\otimes e_i,\ f\in \Omega^0.$
\end{enumerate}
\end{lemma}

\begin{proof}  It is easy to see that elements of types (1),(2),(3) are contained in the kernel of $\zeta$. Let us now prove that 
any element in the kernel of $\zeta$ is a linear combination of elements of the form (1),(2),(3). 
Let $\alpha=\sum\limits_{i}\alpha_i\otimes e_i\in $ ker $\zeta$. We may assume that $\alpha$ is homogeneous (i.e. $\alpha\in \Omega^k\otimes \Bbb C^n$ for some $k$). We have $\zeta(\alpha)=\sum\limits_{i} d\alpha_i \wedge dx_i=0$, so $w=\sum\limits_{i} \alpha_i\wedge dx_i$ is closed and thus $w$ is exact. 
Let us first assume that $k>0$. Then by the Poincar\'e lemma, we have $w=d(\sum\limits_{i} b_i\wedge dx_i)$ for some $b_i\in \Omega$. Then $\sum\limits_{i}( \alpha_i-db_i)\wedge dx_i=0 $. Therefore, by subtracting from $\alpha$ elements of type (1), we can replace $\alpha_i-db_i$ by $\alpha_i$ and assume without loss of generality that $\sum\limits_{i}\alpha_i\wedge dx_i=0.$ By standard linear algebra (exactness of the Koszul complex), this means that $\alpha_i=\sum\limits_{j}\alpha_{ij}\wedge dx_j, \ \alpha_{ij}=\alpha_{ji}$. So $$\alpha=\sum\limits_{i,j} \alpha_{ij}\wedge dx_j\otimes e_i=\sum\limits_{i}\alpha_{ii}\wedge dx_i\otimes e_i+\sum\limits_{i<j}(\alpha_{ij}\wedge dx_j\otimes e_i+\alpha_{ij}\wedge dx_i\otimes e_j).$$
This is a linear combination of elements of type (2), as desired. 

Now assume that $k=0$. Then $\sum \limits_{i} d\alpha_i\wedge dx_i=0$, so the 1-form $\sum \alpha _idx_i$ is closed. Hence it is exact, which means $\alpha_i=\frac{\partial f}{\partial x_i}$ for some $f\in \Omega^0$, i.e.  $\alpha$ is an element of type (3).
\end{proof} 

\begin{lemma}
\label{lemsym}
In $A_4$, the images of $[x_1[x_2,x_3],x_4]\in B_2$ are antisymmetric in $x_1,x_2,x_3,x_4$.
\end{lemma}

\begin{proof}
The antisymmetry in $x_2,x_3$ is clear. We show the antisymmetry in $x_1,x_2$. 

$\begin{array}{lcl}
[x_1[x_2,x_3],x_4]&=& [[x_1x_2,x_3],x_4]-[[x_1,x_3]x_2,x_4]\\
&=&[[x_1x_2,x_3],x_4]-[x_2[x_1,x_3],x_4]+[[x_2,[x_1,x_3]],x_4]\\
&=& -[x_2[x_1,x_3],x_4]\ \text{ mod }L_3.
\end{array}$

Now we show the antisymmetry in $x_1,x_4$.

$\begin{array}{lcl}
[x_1[x_2,x_3],x_4]&=&[x_1,x_4][x_2,x_3]+x_1[[x_2,x_3],x_4]\\
&=&[x_1,x_4][x_2,x_3]+[[x_2,x_3],x_1x_4]-[[x_2,x_3],x_1]x_4\\
&=&[x_1,x_4][x_2,x_3]-x_4[[x_2,x_3],x_1]\\
&=&-[x_4[x_2,x_3],x_1]\ \text{ mod }L_3.
\end{array}$
\end{proof}

Now we prove Theorem \ref{fsthm}(3).  Recall that $L_2$ is spanned by elements of the form $[a,x_i],\ a\in A$. Also, we know from Lemma \ref{lemmaBapat1} that $[M_3,A]\subset L_3$. Thus we have a surjective map $\eta :A/M_3\otimes \Bbb{C}^n \to B_2$ defined by $\eta(a\otimes e_i)=[a,x_i].$ This means that we have a surjective map $\eta:\Omega^{\rm even} \otimes \Bbb{C}^n\to B_2$. Then we have the following diagram

\begin{displaymath}
    \xymatrix{
        \Omega^{\rm even}\otimes \Bbb{C}^n \ar[r]^\zeta \ar[d]_\eta & \Omega_{\rm exact}^{\rm even} \\
        B_2 \ar[ur]_{\theta}        }
\end{displaymath}
which is commutative. So to prove the result, it suffices to show that ker$\zeta=\ \text{ker}\eta$. Thus we have to show that any element of ker$\zeta$ is contained in ker$\eta$. By Lemma \ref{lemkerspan}, it suffices to check that elements of type $(1)$ and $(2)$ are killed by $\eta$.

Type $(1)$ is clear, since if $a\in A/M_3$ maps to $\Omega_{\rm exact}^{\rm even}$, then $a\in L_2$ as we showed in Corollary \ref{cor3.4} (i).

Now we show this for type (2). First we show that 
$$
\eta(\beta dx_i\otimes e_i)=0\ \forall \beta.
$$ 
To this end, take $a=b[x_k,x_i].$ We have $\phi(b[x_k,x_i])=\phi(b)\wedge dx_k\wedge dx_i$. Thus we have $\eta(\phi(b)\wedge dx_k\wedge dx_i\otimes e_i)=[b[x_k,x_i],x_i]=0\ \text{mod }L_3$ since it is antisymmetric in the last two variables by Lemma \ref{lemsym}. But each $\beta$ is a linear combination of $\phi(b)\wedge dx_k$. 

Now deal with $\eta(\beta dx_i\otimes e_j+ \beta dx_j\otimes e_i)$ similarly. Take $a=b[x_k,x_j].$ We have $\phi(b[x_k,x_j])=\phi(b)\wedge dx_k\wedge dx_j$, so 
$$
\eta(\phi(b)\wedge dx_k\wedge dx_i\otimes e_j+\phi(b)\wedge dx_k\wedge dx_j\otimes e_i)=
$$
$$
[b[x_k,x_i],x_j]+[b[x_k,x_j],x_i]=0\ \text{mod}\  L_3,
$$ 
again by Lemma \ref{lemsym}. 

It remains to consider the case of elements of type $(3)$. It is enough to take $f=(\sum_j \lambda_jx_j)^m$, 
$\lambda_i\in \Bbb C$, $m\in \Bbb Z_{\ge 1}$, since such elements along with constants span $\Omega^0$. 
Then $\frac{\partial f}{\partial x_i}=m\lambda_i (\sum_j \lambda_jx_j)^{m-1}$. So the corresponding element of type (3) is 
$$
\alpha=m(\sum_j \lambda_jx_j)^{m-1}\otimes \sum_i\lambda_i e_i
$$
Thus $\eta(\alpha)=m[(\sum_j \lambda_jx_j)^{m-1},\sum_i \lambda_ix_i]=0$, as needed. 

\section{Proof of Theorem \ref{thmBapat1}} 

\begin{lemma}\label{le1} (\cite{AJ}, Lemma 3.1) One has 
$$
[u^3,[v,w]] =3[u^2,[uv,w]]-3[u,[u^2v,w]] +\frac{3}{2}[u^2,[v,[u,w]]]-\frac{3}{2}[u,[v,[u^2,w]]] + 
$$
$$
+[u,[u,[u,[v,w]]]]-\frac{3}{2}[u,[u,[v,[u,w]]]] +\frac{3}{2}[u,[v,[u,[u,w]]]].
$$
\end{lemma} 

\begin{proof} Direct computation (by hand or using a computer). 
\end{proof} 

\begin{corollary}\label{cor1} (\cite{AJ}, Corollary 3.2) Let $a,b,c\in A$, and 
$$
S(a,b,c)=\frac{1}{6}(abc+acb+bac+bca+cab+cba)
$$
be the average of 
products of $a,b,c$ in all orders. Then for $m\ge 2$
$$
[S(a,b,c),B_m]\subset [ab,B_m]+[bc,B_m]+[ca,B_m]+[a,B_m]+[b,B_m]+[c,B_m]
$$
inside $B_{m+1}$. 
\end{corollary}  

\begin{proof} In Lemma \ref{le1}, set $u=t_1a+t_2b+t_3c$, $v$ to be any element of $A$, and $w$ to be any element of $L_{m-1}$, and take the coefficient of $t_1t_2t_3$. 
\end{proof} 

Now we are ready to prove the theorem. We may assume that $A=A_n$, generated by $x_1,...,x_n$. 
Let $P$ be a monomial in $x_i$. We have to show that $[P,B_m]$ is contained in 
the sum of $[x_i,B_m]$, $[x_ix_j,B_m]$, and $[x_i[x_j,x_k],B_m]$ (the latter with $i,j,k$ distinct). Since by Proposition \ref{propfeigin}, 
$[M_3,L_j]\subset L_{j+2}$, we may view $P$ as an element of $A/M_3$. By Theorem \ref{fsthm}, 
$A/M_3\cong \Omega^{\rm even}_*$, so we may view $P$ as an element of $\Omega^{\rm even}$. 

Let $E$ be the linear span of elements of the form $S(a,b,c)$, where $a,b,c\in \Omega^{\rm even}_*$ 
are of positive degree. Let $X$ be the span of $1,x_i,x_ix_j$ and $x_idx_j\wedge dx_k$ with $i,j,k$ distinct. 

\begin{lemma}\label{le2} (\cite{AJ}) One has $\Omega^{\rm even}_*=E+X+\Omega^{\rm even}_{\rm exact}$.
\end{lemma} 

\begin{proof} Note that it suffices to check this for the associated graded algebra of $\Omega^{\rm even}_*$ under the filtration by rank of forms, i.e. for $\Omega^{\rm even}$ under ordinary multiplications. Note that $\Omega^{\rm even}$ is a commutative algebra, so in this algebra $S(a,b,c)=abc$. Also, $\Omega^{\rm even}$ is generated by $x_i$ and $dx_j\wedge dx_k$. 
Thus, $\Omega^{\rm even}/E$ is spanned by $1,x_i,x_ix_j,x_idx_j\wedge dx_k,dx_j\wedge dx_k$, and $dx_i\wedge dx_j\wedge dx_k\wedge dx_l$. 
But the last two forms are exact, as is $x_idx_i\wedge dx_k$. This implies the lemma. 
\end{proof} 
 
Now we prove by induction in degree of $P\in \Omega^{\rm even}_*$ that 
$$
[P,B_m]\subset \sum_i [x_i,B_m]+\sum_{i,j}[x_ix_j,B_m]+\sum_{i,j,l}[x_i[x_j,x_k],B_m].
$$ 
The base of induction is obvious. The induction step follows from Lemma \ref{le2}, Theorem \ref{fsthm}(4) and Corollary \ref{cor1}. This implies the theorem. 

\section{Application of representation theory of $W_n$ to the study of $B_i$ and $N_i$ and proof of Theorem \ref{hilserthm}}

Let us explain how the representation theory of $W_n$ can be used to study $B_i$ and $ N_i$.

\begin{proposition}
If $i\geq 3$ then $\Omega _{\rm closed}^ N(0\leq N\leq n)$ does not occur as a composition factor in $B_i$ nor $N_i$.
\end{proposition}

\begin{proof}
First, consider $N=n$, and look at the polylinear part in $x_i$. In $A_n$ this part is the regular representation of $S_n$, so contains a single copy of the sign representation of $S_n$. But this copy already occurs in $\overline{B}_1\oplus B_2$ and $N_1\oplus N_2$, so does not occur in $\underset{i\geq 3}{\oplus} B_i$ nor $\underset{i\geq 3}{\oplus} N_i$. Hence, $\Omega_{\rm closed}^N$ cannot occur in these, as it contains such a copy.

Now, for $N<n$, $\Omega_{\rm closed}^N$ still cannot occur, as it contains a vector $dx_1\wedge\cdots \wedge dx_N$ in degree $\underbrace{(1,\cdots,1}_{N},0\cdots,0)$, which leads to a contradiction with the above if we mod out the other variables ${x_{N+1},\cdots,x_n}$. 
\end{proof} 

Thus, $B_i,\ N_i\ (i\geq3)$ are equal in the Grothendieck group of $\mathcal{C}$ to $\oplus_\lambda \mathcal{F}_\lambda$. 
So, for the multivariable Hilbert series (in which the power of the variable $t_i$ counts the degree with respect to $x_i$) we have $$h_{B_i}(t_1,\cdots,t_n)=\sum\limits_{\lambda}m_\lambda h_{\mathcal{F}_\lambda}(t_1,\cdots,t_n)=\frac{\sum\limits_{\lambda}m_\lambda h_{S_\lambda(V)}(t_1,\cdots,t_n)}{(1-t_1)\cdots (1-t_n)}.$$
An analogous formula holds for $N_i$. Moreover, by Theorem \ref{finlength}, 
$B_i$ and $N_i$ have finite length, i.e. the sum over $\lambda$ is finite, i.e.  the numerator is a polynomial. 
This proves Theorem \ref{hilserthm} (by setting $t_i=t$ for all $i$). 

For $n=2$ we know from \cite{AJ} that $|\lambda|\leq 2m-3$ for $B_m$ and from \cite{Ke} that $|\lambda|\leq 2m-2$ for $N_m$. Therefore,
$$
h_{B_i}=\frac{\sum\limits_{\lambda}m_\lambda(t_1^{\lambda_1}
t_2^{\lambda_2}+t_1^{\lambda_1-1}
t_2^{\lambda_2+1}+\cdots+t_1^{\lambda_2}
t_2^{\lambda_1})}{(1-t_1)(1-t_2)}
,$$
with $|\lambda|\leq 2m-3$, and $h_{N_i}$ is given by a similar formula with $|\lambda|\leq 2m-2$.

This means that it suffices to know $h_{B_i}$ for degrees $\leq 2m-3$ and $h_{N_i}$ for degrees  $\leq 2m-2$. This allows one to compute $h_{B_i} (i\leq7)$ by using computer, 
which is done in \cite{AJ}.

\section{Lower central series of algebras with relations} 

Consider now the lower central series of algebras with relations. In this case, much less is known than for free algebras, and we will discuss some theoretical and experimental results and conjectures. Specifically, consider the algebra $A:=\Bbb C<x,y>/<P>$, where $P$ is a noncommutative polynomial 
of $x,y$ of some degree $d$ with square-free abelianization (i.e.  the corresponding commutative polynomial factors into distinct linear factors). In this case, $A_{ab}$ is the function algebra on a union of $d$ lines, so by Theorem \ref{JorOr}, $B_m$ and $N_m$ are finite dimensional for $m\ge 2$. This means that $B_m[r]=0$ and $N_m[r]=0$ for sufficiently large $r$. The following theorem gives a bound for how large 
$r$ should be. 

\begin{proposition}\label{vanish} (i) For $m\ge 2$ we have $B_m[r]=0$ if $r\ge 2d+2m-5$.  

(ii) For $m\ge 2$ we have $N_m[r]=0$ if $r\ge 2d+2m-4$. 
\end{proposition} 

\begin{proof} 
(i) By \cite{KL}, Theorem 1, the Hilbert series of $B_2$ is $h_{B_2}(t)=t^2(1+t+...+t^{d-2})^2$. 
The degree of this polynomial is $2d-2$, so we have $B_2[r]=0$ for $r\ge 2d-1$. Now, by Theorem 1.3(4) 
of \cite{AJ} (Theorem \ref{thmBapat2} for $m=2$), we have 
$$
B_{m+1}=[x,B_m]+[y,B_m]+[xy,B_m].
$$ 
This implies that if $B_m[r]=0$ for $r\ge s$ then $B_{m+1}[r]=0$ if $r\ge s+2$. 
Thus, arguing by induction starting from $m=2$, we get that $B_m[r]=0$ for $r\ge 2d-1+2(m-2)=2d+2m-5$. 

(ii) This follows from (i) and Theorem 1.2 of \cite{Ke}, which implies that $N_i=xB_i+yB_i$. 
\end{proof} 

Now consider the case when $P$ is ``Weil generic", i.e.  ``outside of a countable union of hypersurfaces" (in the space of noncommutative polynomials 
of degree $d$). In this case, it is clear from Chevalley's constructibility theorem that $\dim B_m[r]$ and $\dim N_m[r]$ are independent of $P$. 
Proposition \ref{vanish} says that the width of the interval of nonzero values for $\dim B_m[r]$ is at most $2d+m-5$ (these values may occur for $m\le r\le 2d+2m-6$). 
However, computer calculations for small $d$ show that this bound is not sharp, and the width does not actually increase with $m$ for fixed $d$; for
instance, it appears to be $\le 3$ for $d=3$ and $\le 5$ for $d=4$. For this reason, on the basis of computational evidence we make the following conjecture: 

\begin{conjecture} For Weil generic $P$ of degree $d$, we have $B_m[r]=N_m[r]=0$ for $r\ge 2d+m-3$. 
In other words, the width of the interval of nonzero values is $\le 2d-3$. 
\end{conjecture} 

For instance, for $d=3$ (the smallest nontrivial case), the prediction is that $B_m[r]=0$ for $r\ge m+3$. 

Now consider the structure of $B_m$ in more detail. Let $FL_2$ be the free Lie algebra in generators $x,y$. 
We have a natural Lie algebra homomorphism $\widetilde{\psi}: FL_2\to A$. 
Observe that $B_m[m]=L_m[m]=\widetilde{\psi}(FL_2[m])$. In degree $\ge 2$, 
the image of this homomorphism is contained in $[A,A]$. Thus, we have a homomorphism of Lie algebras
$\psi: FL_2^{\ge 2}\to [A,A]$, whose image in degree $m$ is $L_m[m]=B_m[m]$. 

Now, the Hilbert series of $FL_2^{\ge 2}$ and of $[A,A]$ can be computed explicitly. 
For the former, it is obtained in a standard way from the PBW theorem: 
$$
h_{FL_2^{\ge 2}}(t)=\sum_{i\ge 2}a_it^i,\text{ where }\prod_{i\ge 2}(1-t^i)^{a_i}=\frac{1-2t}{(1-t)^2}.
$$
For the latter, we have 
$$
h_{[A,A]}(t)=\sum_{i\ge 2}c_it^i=h_A(t)-h_{A/[A,A]}(t),
$$
and we have 
$$
h_A(t)=\frac{1}{1-2t+t^d}
$$
(see \cite{EG}, Theorem 3.2.4; the term $t^d$ accounts for the relation of degree $d$), while 
$$
h_{A/[A,A]}(t)=1+\sum_{i\ge 1}b_it^i,\text{ where }\prod_{i=1}^\infty (1-t^i)^{b_i}=\prod_{s=1}^\infty (1-2t^s+t^{ds})
$$
(\cite{EG}, Theorem 3.7.7). 
This implies that $\lim_{n\to \infty}a_n^{1/n}=2$, while 
$\lim_{n\to \infty}c_n^{1/n}=\delta^{-1}$, where $\delta$ is the smallest positive root of the equation 
$1-2t+t^d=0$ (clearly, $\delta^{-1}<2$). So we have that $c_n<a_n$ for large enough $n$, hence
$\psi$ is not injective starting from some degree. This gives rise to the following question.

\begin{question}\label{ques} 
Is $\psi$ surjective in some degree $m$? 
\end{question} 

Note that surjectivity of $\psi$ in two consecutive degrees implies strong consequences 
about the structure of the lower central series. Namely, we have the following proposition. 

\begin{proposition}\label{twocon} Let $A=A_n(\Bbb C)/I$, where $I$ is a homogeneous ideal. 
Let $q\ge 3$, and suppose that $B_2[m]=0$ for $m\ge q$, and for some $m\ge q+1$, 
the natural map $\psi: FL_n^{\ge 2}\to [A,A]$ is surjective in degrees $m$ and $m-1$.
Then $[A,A][\ell]=L_\ell[\ell]=B_\ell[\ell]$ and $B_s[\ell]=0$ for $2\le s\le \ell-1$ 
for all $\ell\ge m-1$.  
\end{proposition} 

\begin{proof} 
Note that if $\psi$ is surjective in some degree $m\ge 3$ then $[A,A][m]=L_m[m]$, so 
in particular $B_i[m]=0$ for $i=2,...,m-1$, and we have $[z_{m-2},...,z_1,x_ix_j]\in L_m[m]$ if each
$z_p$ is one of the generators $x_i$. Hence, using \cite{BJ}, Corollary 1.5 (Theorem \ref{thmBapat2} above), 
we have 
$$
B_m[m+1]=\sum_i [x_i,B_{m-1}[m]]+\sum_{i,j} [x_ix_j,B_{m-1}[m-1]]=
$$
$$
=\sum_{i,j}[x_ix_j,B_{m-1}[m-1]],
$$
since  $B_{m-1}[m]=0$. Now, using Lemma \ref{multilin}, we can put $x_ix_j$ in the innermost slot, 
and get that $B_m[m+1]$ is spanned by linear combinations of $[z_{m-1},...,z_1,x_ix_j]$, where each $z_p$ 
is one of the generators $x_i$. 
But as explained above, $[z_{m-2},...,z_1,x_ix_j]\in L_m[m]$, so we get that $B_m[m+1]=0$. 

Now, if $\psi$ is also surjective in degree $m-1$, then we get $B_i[m-1]=0$ for $i=2,...,m-2$, 
so for any $3\le i\le m-1$, we have (using \cite{BJ}, Corollary 1.5 again): 
$$
B_s[m+1]=\sum_i [x_i,B_{s-1}[m]]+\sum_{i,j}[x_ix_j,B_{s-1}[m-1]]=0.
$$
So we get that if $m\ge q+1$ and $\psi$ is surjective in degrees $m$ and $m-1$ then 
it is also surjective in degree $m+1$ (and thus by induction in all degrees $\ell\ge m-1$). 
This implies the statement. 
\end{proof} 

We don't know, however, if the answer to Question \ref{ques} is positive, and in fact 
this hope is not supported by computational evidence. Let us consider the case $d=3$, and 
$A=A_2/<P>$, where $P$ is a Weyl generic element of degree $3$. In this case, the smallest 
degree where $\psi$ has a chance of being surjective (i.e.  the smallest $m$ for which 
$a_m\ge c_m$) is $m=16$. Namely, we have $a_{16}=4080$, while $c_{16}=4036$. 
Computation shows, however, that the map $\psi$ in degree $16$ is not surjective (even though surjectivity is possible dimensionwise): 
it has rank $4031$ and a $5$-dimensional cokernel. 
A similar pattern occurs in degrees $17,18,19$: while $a_n>c_n$ in these degrees, 
the map $\psi$ has a nonzero cokernel of dimensions $4,5,4$. Namely, we have: 
$$
a_{17}=7710,\ c_{17}=6552,\ {\rm rank}\psi[17]=6548;
$$
$$
a_{18}=14532,\ c_{18}=10615,\ {\rm rank}\psi[18]=10610;
$$
$$
a_{19}=27594,\ c_{19}=17216,\ {\rm rank}\psi[19]=17212.
$$
Thus, instead of eventually vanishing, the dimensions of 
the cokernels of $\psi$ seem to stabilize to the pattern $5,4,5,4,...$ 
We do not know an explanation for this phenomenon, and do not know if it continues beyond degree $19$. 

We note, however, that according to our computations, it appears that $\psi$ is eventually surjective (i.e.   
we are in the setting of Proposition \ref{twocon}) in the case when $A=A_2/<P,Q>$, where $P$ is Weil generic of degree $3$ 
and $Q$ is Weil generic of degree $8$. In this case, the smallest degree in which surjectivity of $\psi$ is possible dimensionwise is $15$, and a computer calculation shows that 
$\psi$ is indeed surjective in degrees $15$ and $16$ (of ranks $1974$, $3045$, respectively). Thus Proposition \ref{twocon} applies for degrees
$m\ge 15$. In lower degrees, the dimensions of $B_m[i]$ can be easily computed by a computer algebra system;
thus, one can get a complete list of dimensions of $B_m[i]$ in this case. 

\begin{remark} Note that $8$ is the smallest integer $n$ such that the series 
$(1-2t+t^3+t^n)^{-1}$ has positive coefficients, and therefore the algebra $A$ is infinite dimensional by the Golod-Shafarevich inequality, 
\cite{GS}.
\end{remark} 

\begin{remark}
These computations were done using a MAGMA program written by Eric Rains. 
It computes over a large finite field with randomly chosen relations of the given degrees. 
Thus, the computational results of this subsection should be viewed as conjectural.  
\end{remark}

\end{document}